\documentclass[11pt]{article}
\usepackage[top=1in,bottom=1in,left=1in,right=1in,a4paper]{geometry}
\usepackage{setspace}

\usepackage{graphicx}

\usepackage{float}
\usepackage{color}
\usepackage{enumitem}
\usepackage{setspace}
\usepackage{subcaption}
\usepackage{caption}
\usepackage[linktoc=all,hidelinks]{hyperref}

\usepackage{amssymb}
\usepackage{amsmath}
\usepackage{amsthm}
\usepackage{mathrsfs}

\usepackage{mathtools}
\usepackage{authblk}

\newtheoremstyle{Break}	% name
  {\topsep}		%Space above, empty = `usual value'
  {\topsep}		%Space below
  {}			%Body font
  {}			%Indent amount (empty = no indent, \parindent = para indent)
  {\bfseries}	%Thm head font
  {.}			%Punctuation after thm head
  {\newline}	%Space after thm head: \newline = linebreak
  {}			%Thm head spec

\newtheorem*{Thm*}{Theorem}
\newtheorem{Thm}{Theorem}

\newtheorem*{Pb*}{Problem}

\theoremstyle{Break}
\newtheorem*{Rm*}{Remark}

\title{Random Reordering in SOR-Type Methods}
\author[1]{Peter Oswald \thanks{agp.oswald@gmail.com}}
\author[2]{Weiqi Zhou \thanks{weiqizhou@mathematik.uni-marburg.de}}
\affil[1]{\small Institute for Numerical Simulation, University of Bonn}
\affil[2]{\small FB Mathematics and Informatics, Philipps-University Marburg}
\date{}							

\begin{document}

\maketitle

\begin{abstract}
\noindent
When iteratively solving linear systems $By=b$ with Hermitian positive semi-definite $B$, and in particular when solving least-squares problems for $Ax=b$ by reformulating them as $AA^\ast y=b$, it is often observed that SOR type methods (Gau\ss-Seidel, Kaczmarz) perform suboptimally for the given equation ordering, and that random reordering improves the situation on average. This paper is an attempt to provide some additional theoretical support for this phenomenon. We show error bounds for two randomized versions, called shuffled and preshuffled SOR, that improve asymptotically upon the best known bounds for SOR with cyclic ordering. Our results are based on studying the behavior of the triangular truncation of Hermitian matrices with respect to their permutations.\\[1ex]
{\bf Keywords}: SOR method, Kaczmarz method, random ordering, triangular truncation, convergence estimates.\\[1ex]
{\bf 2010 MSC}: 65F10, 15A60.
\end{abstract}

\section{Introduction}\label{sec1}

In this paper, we discuss the influence of the equation ordering in a linear system $By=b$ on deriving upper bounds for the convergence speed of the classical successive over-relaxation (SOR) method. We assume that $B$ is a complex $n\times n$ Hermitian positive semi-definite matrix with positive diagonal part $D$. If we write $B=L+D+L^{\ast}$, where $L$ denotes the strictly lower triangular part of $B$ and $\ast$ stands for Hermitian conjugation, then one step of the classical SOR iteration reads
\begin{equation}\label{SOR}
y^{(k+1)}=y^{(k)}+\omega (D+\omega L)^{-1}(b-By^{(k)}), \qquad k=0,1,\ldots.
\end{equation}
The classical Gau\ss-Seidel method for solving $By=b$ emerges if one takes $\omega=1$. If one attempts to solve a general linear system $Ax=b$ in the least-squares sense, then one has the choice to apply the SOR method to either the normal equation
$A^{\ast} A x = A^{\ast} b$ or to $AA^{\ast} y = b$. In the latter case, the algorithm resulting from applying (\ref{SOR}) to $B=AA^\ast$ is equivalent to the Kaczmarz method (here approximations to the solution of $Ax=b$ are recovered by setting $x^{(k)}=A^{\ast} y^{(k)}$), see \cite{hackbusch1994iterative}. 
\\[1ex]
To make the paper more readable and avoid technical detail, we make two additional assumptions. First,   
we consider only consistent systems ($b\in \mathrm{Ran}(B)$). This guarantees convergence of (\ref{SOR}) for any $0<\omega<2$ and any $y^{(0)}$ to a solution of $By=b$, while for inconsistent systems the method diverges (this does not contradict the known convergence of the Kaczmarz method for inconsistent systems
$Ax=b$ since the divergence manifests itself only in the $\mathrm{Ker}(B)=\mathrm{Ker}(AA^\ast)$ component of $y^{(k)}$ which is annihilated when recovering $x^{(k)}=A^\ast y^{(k)}$). Secondly, we assume that $B$ has unit diagonal ($D$=$I$) which can always be achieved by transforming to the equivalent rescaled system $D^{-1/2}BD^{-1/2}\tilde{y}=D^{-1/2}b$ (for the Kaczmarz algorithm, one would simply use row normalization in $A$). Alternatively, the analysis of the SOR method can be carried out with arbitrary $D>0$, with minor changes in some places, see \cite{oswald2015convergence} for some details. With both approaches, $D$ enters the final results via the spectral properties of the transformed $B$ or its norm, respectively. Note that with $D=I$, one step of (\ref{SOR}) consists of $n$ consecutive projection steps onto the $i$-th coordinate direction, $i=1,2,\ldots,n$, and the method thus becomes an instance of the alternating direction method (ADM). 
Unless stated otherwise, these two assumptions are silently assumed from now on.
\\[1ex]
Since any positive semi-definite $B$ can be factored, in a non-unique way, as 
$$
B=AA^{\ast}, 
$$
we can always assume that $B$ is produced by some $n\times m$ matrix $A$ with unit norm rows. Denote by  $r=\mathrm{rank}(B)\le \min(n,m)$ its rank, the spectral properties of $A$ and $B$ are obviously related: The non-zero eigenvalues $\lambda_1 \ge \lambda_2\ge \ldots\ge\lambda_r$ of $B$ are given by the squares of the non-zero singular values of $A$. Thus, if we define the essential condition number $\bar{\kappa}(A)$ of a matrix as the quotient of its largest and smallest non-zero singular values then
%!end%
$$
\bar{\kappa}:=\bar{\kappa}(B)=\bar{\kappa}(A)^2=\frac{\lambda_1}{\lambda_r}.
$$
The unit diagonal assumption $D=I$ for $B$ implies $0<\lambda_r\le 1\le \lambda_1\le n$.
In the convergence analysis below, we will use the energy semi-norm $|y|_B = \langle By,y\rangle^{1/2}=\|A^{\ast} y\|^{1/2}$ associated with $B$, it is a norm if and only if $B$ is non-singular, i.e., positive definite. Here, $\langle \cdot,\cdot\rangle$ and $\|\cdot\|$ denote the usual Euclidian scalar product and norm
in $\mathbb{C}^n$, respectively. Later, we will use the notation $\|\cdot\|$
also for matrices (then it stands for their spectral norm) which should be clear from the context and not lead to any confusion.
\\[1ex]
Condition numbers and other spectral properties often enter the asymptotic error estimates of iterative schemes for solving linear systems, the best known examples are the standard bounds for the Jacobi-Richardson and conjugate gradient methods for systems with positive definite $B$, see e.g., \cite{hackbusch1994iterative}. For the SOR method, such upper estimates have been established in \cite{oswald1994convergence} for non-singular $B$, and recently improved in
\cite{oswald2015convergence} within the framework of the Kaczmarz iteration to include the semi-definite case:

\begin{Thm}\label{theo1}
Let $B$ be a given $n\times n$ Hermitian positive semi-definite matrix with unit diagonal, and assume that $By=b$ is consistent, i.e., possesses at least one solution $\bar{y}$. Then the SOR iteration (\ref{SOR}) converges for $0<\omega<2$ in the energy semi-norm associated with $B$ according to
\begin{equation}\label{SORest}
|\bar{y}-y^{(k)}|_B^2 \le (1-\frac{(2-\omega)\omega \lambda_1}{(1+\frac12\lfloor\log_2(2n)\rfloor \omega \lambda_1)^2\bar{\kappa}})^k |\bar{y}-y^{(0)}|_B^2, \qquad k\ge 1.
\end{equation}
If $B$ is singular, then for sufficiently small rank $r$ the term $\frac12\log_2(2n)$ can be replaced by the smaller term $C_0\ln r$, where $C_0$ is an absolute constant. 
\end{Thm} 

The proof of (\ref{SORest}) rests on rewriting the squared energy semi-norm of $Q y$, where
$$
Q=I-\omega(I+\omega L)^{-1}B
$$
is the error iteration matrix associated with (\ref{SOR}), as
$$
|Q y|^2_B =|y|_B^2 - \omega(2-\omega)\|(I+\omega L)^{-1}By\|^2\le |y|_B^2 - \frac{\omega(2-\omega)\|By\|^2}{\|I+\omega L\|^2}.
$$
and using a spectral norm inequality for $L$ from  \cite{oswald1994convergence},
\begin{equation}\label{LOGest}
\|L\| \le \frac12\lfloor\log_2(2n)\rfloor \|B\|,
\end{equation}
to estimate the term $\|I+\omega L\|\le 1+\omega \|L\|$. For singular $B$ with small rank $r<n$, the estimate (\ref{LOGest}) has been improved in
\cite{oswald2015convergence} to 
\begin{equation}\label{LOGestr}
\|L\| \le C_0 \ln r \|B\|, \quad r\ge 2,
\end{equation}
where $C_0$ is a fixed positive constant. It is well known that the estimate (\ref{LOGest}) is sharp in its logarithmic dependency on $n$, more precisely
$$
b_n:=\sup_{B\neq 0} \frac{\|L\|}{\|B\|} \asymp \frac{1}{\pi}\ln n,\qquad n\to \infty,
$$
where the supremum is taken with respect to all $n\times n$ matrices $B$, and $\asymp$ stands for asymptotic equality (see \cite{davidson1988nest, mathias1993hadamard} for sharp estimates and examples). Similar lower estimates hold also for Hermitian positive semi-definite matrices $B$ with unit diagonal $D=I$,
and examples exist that the necessity of the logarithmic terms in (\ref{SORest}), see \cite{oswald1994convergence}.
\\[1ex]
For non-singular $B$, i.e., when $|\cdot|_B$ becomes a norm and the system has full rank $r=n$, the outlined idea of proof for Theorem \ref{theo1} has been carried out in detail in \cite{oswald1994convergence}. The changes for singular $B$ are minimal, the proof of (\ref{LOGestr}) for this case can be found in \cite[Theorem 4]{oswald2015convergence}, see also the proof of Part b) of Theorem \ref{ThmSOR} in Section \ref{sec3}. 
\\[1ex]
The crucial inequalities (\ref{LOGest}) and (\ref{LOGestr}), and consequently the error bounds in Theorem \ref{theo1}, suffer from one serious drawback:
They are invariant under simultaneously reordering rows and columns in $B=AA^\ast$ resp. reordering the rows in $A$. Indeed, $B_{\sigma} =P_{\sigma} B P^{\ast}_{\sigma}$ has the same spectrum and spectral norm as $B$ for any permutation $\sigma$ of the index set $\{1,\ldots,n\}$ ($P_{\sigma}$ denotes the associated $n\times n$ row permutation matrix), while the spectral properties of the lower triangular part $L_{\sigma}$  of $B_{\sigma}$ depend on $\sigma$. As a matter of fact, in practice it is often observed (for example, see \cite{young1954iterative,varga1959orderings,feichtinger1994theory}) that reordering improves the convergence behavior of SOR methods as well as other, more general, alternating directions, subspace correction, and projection onto convex sets (POCS) methods. The interest in explaining this observation theoretically has been further stimulated by convergence results for a randomized Kaczmarz iteration in \cite{strohmer2009randomized}. In the language of SOR for solving a consistent system $By=b$ with $D=I$, instead of performing the 
$n$ consecutive projection steps on the $i$-th coordinate that compose the SOR iteration step (\ref{SOR}) in the fixed order $i=1,2,\ldots,n$, the method in \cite{strohmer2009randomized} performs the projection steps on coordinate directions 
by randomly selecting $i$ uniformly and independently from $\{1,\ldots,n\}$ in each single step. For a fair comparison with the original SOR iteration (\ref{SOR}), and the randomized SOR methods discussed below, it is appropriate to combine
$n$ single projection steps on randomly and independently chosen coordinate directions into one iteration step. The iterates 
$y^{(k)}$ of this method which we call for short {\it single step randomized SOR iteration} are now random variables. Under the same assumptions as in Theorem \ref{theo1}, the following estimate for the expectation of the squared energy semi-norm error
can be deduced from \cite{strohmer2009randomized}:
\begin{equation}\label{SVest}
\mathbb{E}(|{y}^{(k)}-y^\ast|_B^2) \le \left(1-\frac{(2-\omega)\omega\lambda_1}{n\bar{\kappa}}\right)^{kn}
|{y}^{(0)}-y^\ast|_B^2,\qquad k\ge 1.
\end{equation}
The two upper estimates (\ref{SORest}) and (\ref{SVest}) are obtained by different techniques, and although a rough comparison of
the upper bounds suggests that the single step randomized SOR beats the original SOR, in practice this is generally not true, and depends on
the given system and the ordering of the equations in it.
\\[1ex]
In this paper, we consider two different randomization strategies for SOR closer to the original method. In the first, given the $k$-th iterate $y^{(k)}$, we choose (independently and randomly) a permutation $\sigma$ of $\{1,\ldots,n\}$, and do one full SOR iteration step (\ref{SOR}) with $By=b$ and $y^{(k)}$ replaced by $B_{\sigma} y_{\sigma} = b_{\sigma}$ and $y^{(k)}_\sigma =P_\sigma y^{(k)}$, where $y_{\sigma} =P_{\sigma} y$, $b_{\sigma} = P_{\sigma} b$. Then the original order is restored by setting $y^{(k+1)} =P^\ast_\sigma y^{(k+1)}_\sigma$.  This approach which we call for short {\it shuffled SOR iteration} is equivalent to a random ordering without repetition in each sweep of $n$ steps of the single step randomized SOR iteration. In practice, random ordering without repetition is considered superior to random ordering with repetition although theoretical proof for this observation is yet missing, see the conjectures in \cite{recht2012beneath,duchi2012}. In \cite{wright2015}, where the counterpart of the shuffled SOR iteration for coordinate descent methods in convex optimization appears as
algorithm EPOCHS, similar statements can be found.
\\[1ex]
It is also tempting to investigate the effect of a one-time reordering, followed by the application of the SOR iteration in the classical, cyclic fashion (we call this {\it preshuffled SOR iteration}). In other words, the preshuffled SOR iteration coincides with a shuffled SOR iteration if we reuse the randomly generated $\sigma$ from the iteration step at $k=0$ for all further iteration steps at $k>0$. Observe that in terms of the Kaczmarz iteration these two schemes merely correspond to shuffling the rows in the row-normalized matrix $A$, i.e., $Ax=b$ is replaced by $P_{\sigma} Ax=P_{\sigma} b$. The numerical experiments presented in \cite{oswald2015convergence} suggest that shuffled and preshuffled SOR iterations often perform in expectation equally good, and better than the single step randomized iteration.
\\[1ex]
The present paper is an attempt to gain some insight into what can be expected from these randomization strategies.  Speaking in mathematical terms, if 
$$
Q_{\sigma}=(I+\omega L_{\sigma})^{-1}((1-\omega)I-\omega L^{\ast}_{\sigma}),
$$ 
denotes the error iteration matrix of the SOR method applied to $B_{\sigma} y_{\sigma}=b_{\sigma}$, then we aim at investigating the quantity
\begin{equation} \label{ExpVal2}
\mathbb{E}[|Qy|_{B}^2]:=\frac{1}{n!}\sum_{\sigma} |Q_{\sigma} y_{\sigma}|_{B_{\sigma}}^2, \qquad |y|_B=1,
\end{equation}
to obtain upper bounds for the expected square energy semi-norm error in the shuffled SOR iteration. 
\\[1ex]
As was outlined above, obtaining estimates for the norm behavior of $Q_\sigma$, and of relevant averages such as (\ref{ExpVal2}), must be closely related to studying the behavior of $L_{\sigma}$ which will be at the heart of our considerations in Section \ref{sec2}. In particular, we apply a corollary of the recently proved paving conjecture to show that for any positive semi-definite $B$ with $D=I$ there is a permutation $\sigma$ (depending on $B$) with the property
\begin{equation}\label{KSL}
\|L_{\sigma}\| \le C_1\|B\|,
\end{equation}
where $C_1$ is an absolute constant.
We further establish that 
\begin{equation} \label{EqLU}
\|\mathbb{E}[LL^{\ast}]\|<\|B\|^2,
\qquad \mathbb{E}[LL^{\ast}]:=\frac{1}{n!}\sum_{\sigma} P_{\sigma}^{\ast}L_{\sigma}L_{\sigma}^{\ast}P_{\sigma},
\end{equation}
which will lead to bounds for (\ref{ExpVal2}). 
\\[1ex] 
In Section \ref{sec3}, we apply the results of Section \ref{sec2} to establish two new error decay bounds for the above mentioned shuffled SOR iterations. First of all, we show that the quantity in (\ref{ExpVal2}) satisfies
$$ 
\mathbb{E}(|Qy|_{B}^2) \le (1-\frac{(2-\omega)\omega \lambda_1}{(1+\omega \lambda_1)^2\bar{\kappa}(B)})|y|_B^2,
$$
which implies a bound for the expected square energy semi-norm error decay of the shuffled SOR iteration that compares favorably with the bounds in
Theorem \ref{theo1}, as the logarithmic dependence on $n$ and $r$ is removed. Next, we prove using (\ref{KSL}) that there exists a $\sigma$ such that the preshuffled SOR iteration can achieve the same effect, i.e., replacing the $\frac12\lfloor\log_2(2(n-1)\rfloor$ resp. $C_0\ln r$ factor by the constant $C_1$ from (\ref{KSL}).
Although asymptotic in nature, and in case of the preshuffled SOR iteration due to the currently available estimates for $C_1$ not yet practical, the bounds established in Theorem \ref{ThmSOR} should be viewed as theoretical support for the numerically observed convergence behavior of shuffled and preshuffled SOR iterations.

\section{Triangular Truncation and Reordering}\label{sec2}

If not stated otherwise, in this section $B=L+D+L^\ast$ belongs to $\mathscr H_n$, the set of all $n\times n$ Hermitian matrices, with no assumptions on positive semi-definiteness or normalization of its diagonal elements (i.e., not assuming $D=I$). The notation of Section \ref{sec1} is reused for this slightly more general situation. 

\begin{Thm} \label{ThmSumM} If  $B\in\mathscr H_n$ then the average operator $\mathbb{E}[LL^{\ast}]$ defined in (\ref{EqLU}) satisfies
$$
\|\mathbb{E}[LL^{\ast}]\|\le 4\|B\|^2.
$$
Moreoever, if $D=I$ and $B$ is positive semi-definite, then (\ref{EqLU}) holds.
\end{Thm}

\begin{proof}
For given $B\in \mathscr H_n$, set $H=L+L^{\ast}$. Since $\|D\|\le \|B\|$, we have
\begin{equation}\label{EqHbyB}
\|H\|=\|B-D\|\le \|B\|+\|D\|\le 2\|B\|,
\end{equation}
while for positive semi-definite $B$
\begin{equation}\label{EqHbyB1}
\|H\|=\|B-I\|\le \max(\|B\|-1,1)\le \|B\|.
\end{equation}
Thus, establishing (\ref{EqLU}) with $B$ replaced by $H$ is enough.
\\[1ex]
Straightforward computation shows that
$$
(P_{\sigma}^{\ast}L_{\sigma} L^{\ast}_{\sigma}P_{\sigma})_{st}=\sum_{k=1}^{\min(s,t)-1}H_{s\sigma(k)}H_{\sigma(k)t},\qquad s,t=1,\ldots,n.
$$ 
By counting the number of permutations for which $\sigma(k)=l$ for some $k=1,\ldots,\min(s,t)-1$ we get
$$
\frac{1}{n!}\sum_{\sigma}(P_{\sigma}^{\ast}L_{\sigma} L^{\ast}_{\sigma}P_{\sigma})_{st}=\frac{(n-1)!}{n!}(\min(s,t)-1)\sum_{l=1}^nH_{sl}H_{lt}=\frac{\min(s,t)-1}{n}(H^2)_{st}.
$$
Hence
$$
\frac{1}{n!}\sum_{\sigma}(P_{\sigma}^{\ast}L_{\sigma} L^{\ast}_{\sigma}P_{\sigma})=\frac{1}{n}K\circ H^2,
$$
where $\circ$ denotes Hadamard multiplication, and 
\begin{equation} 
K=\begin{pmatrix}
0 & 0 & 0 &\ldots & 0 & 0 \\
0 & 1 & 1 &\ldots & 1 & 1 \\
0 & 1 & 2 &\ldots & 2 & 2\\
\vdots & \vdots & \vdots &\ddots &\vdots & \vdots \\
0 & 1 & 2 &\ldots  & n-2 & n-2\\
0 & 1 & 2 &\ldots  & n-2 & n-1 
\end{pmatrix}.
\end{equation}
In other words, the above Hadamard product can be written as the linear combination of $n-1$ diagonally projected submatrices of $H^2$, each of norm $\le \|H^2\|$.
This gives 
$$
\|\frac1{n!} \sum_{\sigma}P^{\ast}_{\sigma} L_{\sigma} L^{\ast}_{\sigma}P_{\sigma}\|\le\frac{n-1}{n}\|H^2\|<\|H\|^2,
$$
which completes the proof.
\end{proof}

The following result was suggested to the first author by B. Kashin (Steklov Institute, Moscow), and is included here with his permission. 

\begin{Thm}\label{ThmMinNorm}
There is an absolute constant $C_2$ such that for any $B\in \mathscr H_n$ there exists a permutation $\sigma$ for which
\begin{equation}\label{MinNorm}
\|L_{\sigma}\| \le C_2\|B\|.
\end{equation}
Moreover, if $D=I$ and $B$ is positive semi-definite then (\ref{KSL}) holds with an absolute constant $C_1\le C_2$.
\end{Thm}
\begin{proof} 
Weaker versions of (\ref{MinNorm}), where the spectral norm $\|L\|=\|L\|_{\ell_2^n\to \ell_2^n}$ is replaced by $\|L\|_{\ell_2^n\to \ell_q^n}$ with $1\le q<2$, have been proved in \cite{kashin1980} and \cite{bourgaintzafriri1987}. 
\\[1ex]
For the proof of (\ref{MinNorm}) we explore the following particular result on matrix paving, which for a long time was known as Anderson's Paving Conjecture
for Hermitian matrices with small diagonal. This conjecture is equivalent to the Kadison-Singer Problem, a positive solution of which was recently given in \cite{marcus2013interlacing}. We formulate it for $B\in \mathscr H_n$  with zero diagonal, and refer to the recent expository paper \cite{casazzatremain2015} for details.
\begin{Thm*}[Anderson's Paving Conjecture] \label{LmPave} 
For any $0<\epsilon<1$, there is an integer $\gamma(\epsilon)\ge 2$ such that for any $n\in\mathbb N$ and any $B\in\mathscr H_n$ with zero diagonal, there exists a partition 
$$
w_1\cup w_2\ldots\cup w_{\gamma}=\{1,2,\ldots,n\},\qquad w_i\cap w_j = \emptyset,\quad i\neq j,
$$
into $\gamma\le \gamma(\epsilon)$ non-empty index subsets such that  
$$
\|B_{w_k}\|\le\epsilon\|B\|, \qquad k=1,\ldots,\gamma.
$$ 
Here $B_{w_k}$ is the $|w_k|\times|w_k|$ submatrix corresponding to the index set $w_k\times w_k$. 
\end{Thm*}

Returning to the proof of Theorem \ref{ThmMinNorm}, by (\ref{EqHbyB}) it is enough to consider matrices $B\in \mathscr{H}_n$ with zero diagonal.
For given $0<\epsilon<1$ we proceed by induction in $n$ to establish (\ref{MinNorm}) with a constant $C_{\epsilon}:=(\gamma(\epsilon)-1)/(1-\epsilon)$, where $\gamma(\epsilon)$ is defined in the above paving theorem. To find an estimate for the best constant $C_2$ in (\ref{MinNorm}), we then optimize with respect to $\epsilon$ resp. $\gamma$.
Any $\sigma$ will do for $n=2$ since  $C_{\epsilon}>1$ and in this case $\|L\|=\|B\|$. Suppose the statement holds for all matrix dimensions less than $n$. For $B\in\mathscr{H}_n$ with zero diagonal, consider the partition $w_1, w_2, \ldots,w_{\gamma}$ of in the above paving theorem, and denote by $\sigma_0$ the permutation that makes $B_{\sigma_0}$ contain the submatrices $B_{w_k}$ as consecutive diagonal blocks, as depicted in Figure \ref{fig1} for $\gamma=3$. Let $B_{l_k}$ be the rectangular submatrices below $B_{w_k}$ in this $B_{\sigma_0}$,  $k=1,\ldots,\gamma-1$.

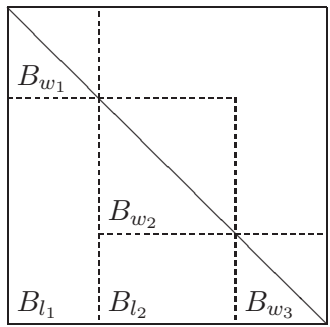
\begin{figure}[H]
\setlength{\unitlength}{0.6cm}
\begin{picture}(8,8)(-9.5, -0.5) \label{fg1}
\multiput(0,0)(7,0){2}{\line(0,1){7}}
\multiput(0,0)(0,7){2}{\line(1,0){7}}
\put(0,7){\line(1,-1){7}}

\multiput(0,5)(0.2,0){10}{\line(1,0){0.1}}
\multiput(2,5)(0,0.2){10}{\line(0,1){0.1}}
\put(0.2,5.3){$B_{w_1}$}

\multiput(2,5)(0.2,0){15}{\line(1,0){0.1}}
\multiput(2,5)(0,-0.2){15}{\line(0,-1){0.1}}
\multiput(2,2)(0.2,0){15}{\line(1,0){0.1}}
\multiput(5,5)(0,-0.2){15}{\line(0,-1){0.1}}
\put(2.2,2.3){$B_{w_2}$}

\multiput(5,2)(0.2,0){10}{\line(1,0){0.1}}
\multiput(5,2)(0,-0.2){10}{\line(0,-1){0.1}}
\put(5.2,0.3){$B_{w_3}$}

\multiput(2,2)(0,-0.2){10}{\line(0,-1){0.1}}
\put(0.2,0.3){$B_{l_1}$}
\put(2.2,0.3){$B_{l_2}$}
\end{picture}
\caption{Block structure of $B_{\sigma}$}\label{fig1}
\end{figure}

For each $k=1,\ldots,\gamma$ we have $|w_k|<n$, and by the induction assumption there exist permutations $\sigma_k$ such that 
$$
\|(L_{w_k})_{\sigma_k}\|\le C_{\epsilon}\|B_{w_k}\|, \qquad k=1,\ldots,\gamma,
$$
where $(L_{w_k})_{\sigma_k}$ is the strictly lower triangular part of $(B_{w_k})_{\sigma_k}$.
\\[1ex]
By superposing the permutations $\sigma_k$ within each block with $\sigma_0$, we get the desired $\sigma$: In each diagonal block of $B_{\sigma}$
we have now $(B_{w_k})_{\sigma_k}$ instead of $B_{w_k}$, and the rectangular submatrices $B'_{l_k}$ below the diagonal blocks are row and column permuted
copies of the previous $B_{l_k}$. 
\\[1ex]
We split $L_{\sigma}$ into the sum of a block-diagonal matrix $L_1$ containing all $(L_{w_k})_{\sigma_k}$, and another 
lower triangular matrix $L_2$
containing all rectangular submatrices $B'_{l_k}$. Since
$$
\|L_1\| \le \max_{k=1,\ldots,\gamma} \|(L_{w_k})_{\sigma_k}\| \le C_{\epsilon} \max_{k=1,\ldots,\gamma} \|B_{w_k}\| \le C_{\epsilon}\epsilon \|B\|,
$$
and 
$$
\|L_2\| \le \sum_{k=1}^{\gamma-1} \|B'_{l_k}\| \le (\gamma-1)\|B\|\le C_{\epsilon}(1-\epsilon)\|B\|
$$
(note that each $B'_{l_k}$ is a row and column permuted version of a rectangular submatrix of the original $B$, thus $\|B'_{l_k}\|\le \|B\|$). Therefore, 
\begin{equation}\label{LMinNorm}
\|L_{\sigma}\|\le \|L_1\|+\|L_2\|\le C_{\epsilon}\|B\|,
\end{equation} 
which concludes the induction step.
\\[1ex]
To find numerical estimates for $C_2$, we need bounds for $\gamma(\epsilon)$. The bounds given in \cite[Section 4]{casazzatremain2015}) are very rough, therefore we rely on Corollary 26 from Tao's blog on the Kadison-Singer problem accessible at \texttt{https://terrytao.wordpress.com/2013/11/04/} which implies the following: For
given $\gamma\ge 2$, there exists a partition into $\gamma^2$ index subsets such that the statement of Theorem \ref{LmPave} holds with 
$\epsilon=\epsilon(\gamma)=2/\gamma +2\sqrt{2/\gamma}$. For $\gamma\ge 12$, one has $\epsilon<1$, and we conclude that
$$
C_2 \le 2\inf_{0<\epsilon<1} C(\epsilon) \le 2\inf_{\gamma\ge 12} C(\epsilon(\gamma)) = 2\min_{\gamma\ge 12} \frac{\gamma^2-1}{1-2/\gamma-2\sqrt{2/\gamma}} =
2907,
$$
with the minimum achieved for $\gamma=18$. The factor $2$ comes from taking into account (\ref{EqHbyB}). This bound is overly pessimistic (note that results closer to the known lower bound $\gamma(\epsilon)\ge 1/\epsilon^2$ would result in much smaller values of $C_2$). 
\\[1ex]
It is therefore worth looking for improvements if $B$ is positive semi-definite and has unit diagonal $D=I$. Then $B-I$ is a Hermitian matrix with zero diagonal and spectrum in $[-1,\|B\|-1]$ satisfying (\ref{EqHbyB1}), and Corollary 25 of Tao's blog yields, for any $\gamma\ge 2$, the existence of a partition into 
$\gamma$ index subsets such that in Theorem \ref{LmPave} we can take $\epsilon=\epsilon'(\gamma)=1/\gamma +2/\sqrt{\gamma}$. Repeating the above proof steps for this case, we see that
$$
C_1 \le \min_{\gamma\ge 6} \frac{\gamma-1}{1-1/\gamma-2/\sqrt{\gamma}} \le 32.42
$$
(here $\epsilon'(\gamma)<1$ for $\gamma\ge 6$, and the minimum is achieved for $\gamma=12$). 
\end{proof}

It remains an open question if an inequality similar to (\ref{MinNorm}) also holds for
the average of the norms $\|L_\sigma\|$, namely if
\begin{equation}\label{ELsigma}
\mathbb{E}[\|L\|] := \frac{1}{n!}\sum_{\sigma} \|L_{\sigma}\|\le C_n\|B\|,\qquad B\in \mathscr H_n,
\end{equation}
holds for some (bounded or slowly increasing) sequence of positive constants $C_n=\mathrm{o}(\ln (n))$ (for a related result, see \cite[Theorem 8.4]{bourgaintzafriri1987}). A proof of (\ref{ELsigma}) would imply improved asymptotic estimates for the expected convergence rate of
the preshuffled SOR iteration, and not only for the best possible convergence rate, as established in Part b) of Theorem  \ref{ThmSOR} below.

\section{Application to SOR Iterations}\label{sec3}
In this section we show \textit{a priori} convergence estimates for the shuffled and preshuffled SOR iterations that improve upon the one for the standard SOR iteration (\ref{SOR}) stated in Theorem \ref{theo1}, at least asymptotically. These estimates are formulated in terms of the energy semi-norm associated with $B$, and are equivalent to estimates in the usual Euclidian norm for the Kaczmarz iteration applied to a consistent linear system $Ax=b$, where $B=AA^{\ast}$. The result is summarized in the following

\begin{Thm}\label{ThmSOR} Let $By=b$ be a consistent linear system with positive semi-definite $B=L+I+L^{\ast}\in \mathscr H_n$, and denote by
$\bar{y}$ an arbitrary solution of it. Fix any $\omega\in (0,2)$.\\[0.5ex]
a) The expected squared energy semi-norm error of the shuffled SOR iteration converges exponentially with the bound
$$
\mathbb{E}(|\bar{y} - y^{(k)}|_B^2)\le \left(1- \frac{\omega(2-\omega)\lambda_1}{(1+\omega\lambda_1)^2\bar{\kappa}} \right)^k |\bar{y}- {y}^{(0)}|_B^2, \quad k\ge 1,
$$
for any $\omega\in (0,2)$.\\[0.5ex]
b)  There exists some ordering $\sigma$ such that the classical SOR iteration on the system $B_{\sigma} y_\sigma =b_{\sigma}$ converges for any $\omega\in (0,2)$with square energy semi-norm error decay
$$
|\bar{y} - y^{(k)}|_B^2\le \left(1- \frac{\omega(2-\omega)\lambda_1}{(1+C_1\omega\lambda_1)^2\bar{\kappa}} \right)^k |\bar{y}- {y}^{(0)}|_B^2, \quad k\ge 1,
$$
where the constant $C_1$ satisfies (\ref{KSL}). 
\end{Thm}

\begin{proof} We start with b). Take the $\sigma$ for which $\|L_{\sigma}\|\le C_1\|B\|$ according to (\ref{KSL}).  To simplify notation, let us drop the subscript $\sigma$ so that now $B_{\sigma}=B=L+I+L^{\ast}$, $b_{\sigma}=b$, $P_{\sigma}=I$, and $\|L\|\le C_1\|B\|$. Recall
the notation $Q =I-\omega(I+\omega L)^{-1}B$ for the error iteration matrix,
and check that $\mathbb{C}^n= U\oplus V$, where $U=\mathrm{Ker}(B)$ and $V=(I+\omega L)^{-1}\mathrm{Ran}(B)$ are $Q$-invariant subspaces (obviously, $Q$ is the identity when restricted to $U$). Write the SOR iterates as
$y^{(k)}=u^{(k)}+v^{(k)}$, $u^{(k)}\in U$, $v^{(k)}\in V$. Since $y^{(k+1)}=Qy^{(k)}+ \omega(I+\omega L)^{-1}b$, by induction it follows that 
\begin{equation}\label{Rec}
u^{(k)}=u^{(0)}, \qquad v^{(k)}=Q^kv^{(0)} + \omega(I+Q+\ldots +Q^{k-1})(I+\omega L)^{-1}b, \qquad k\ge 1.
\end{equation}
Now, any solution $\bar{y}$ of $By=b$ can be written as $\bar{y}=u+\bar{v}$, where $u\in U$ is arbitrary, and $\bar{v}\in V$ is unique.
Because
$$
|\bar{y}-y^{(k)}|_B^2=\langle B(\bar{y}-y^{(k)}),\bar{y}-y^{(k)}\rangle=\langle B(\bar{v}-v^{(k)}),\bar{v}-v^{(k)}\rangle=|\bar{v}-v^{(k)}|_B^2,
$$
and $\bar{v}-v^{(k+1)}=\bar{v}-Qv^{(k)}-\omega(I+\omega L)^{-1}B\bar{v}=Q(\bar{v}-v^{(k)})$, all we need is an estimate of the form
$$
|Qv|_B^2 \le \rho |v|_B^2,\qquad v\in V.
$$
By substituting $\omega B= (I+\omega L)+(I+\omega L^{\ast}) - (2-\omega)I$ below, we get
\begin{eqnarray*}
|Qv|_B^2&=&\langle Bv,v\rangle-\omega\langle B((I+\omega L)^{-1}+(I+\omega L^{\ast})^{-1})Bv,v\rangle\\
&&\qquad\qquad\qquad +\omega\langle B(I+\omega L^{\ast})^{-1}(\omega B)(I+\omega L)^{-1}Bv,v\rangle\\
&=&\langle Bv,v\rangle- \omega(2-\omega)\|(I+\omega L)^{-1}Bv\|^2.
\end{eqnarray*}
Using (\ref{KSL}), the last term can be bounded from below as
$$
\|(I+\omega L)^{-1}Bv\|^2\ge \frac{\|Bv\|^2}{\|I+\omega L\|^2}\ge \frac{\lambda_r|v|_B^2}{(1+\omega C_1\|B\|)^2}=
\frac{\lambda_1|v|_B^2}{(1+\omega C_1\lambda_1)^2\bar{\kappa}}.
$$
Therefore, we obtain 
$$
\rho=1-\frac{\omega(2-\omega)\lambda_1}{(1+\omega C_1\lambda_1)^{2}\bar{\kappa}},
$$
which gives the bound stated in Part b). Moreover, since $\|v\|$ and $|v|_B$ are equivalent norms on $V$, we see that $v^{(k)}\to \bar{v}$. 
According to (\ref{Rec})
$$
y^{(k)} \to u^{(0)} + \bar{v},
$$
so the SOR iteration converges in the usual sense as well, with the $U=\mathrm{Ker}(B)$ component in the limit depending on the starting vector $y^{(0)}$
if $B$ is singular. Returning to the original formulation as preshuffled SOR iteration, the $\mathrm{Ker}(B)$ component of the limit would also depend on $\sigma$.
\\[1ex]
The result of Part a) requires a similar, yet slightly more subtle analysis. Recall that in each step of the shuffled iteration, given the current iterate $y^{(k)}$, we choose a permutation $\sigma$ at random, apply the SOR step (with matrix $B_{\sigma}=P_{\sigma} B P^{\ast}_{\sigma}$ and right-hand side $b_{\sigma}=P_{\sigma} b$) to 
$P_{\sigma} y^{(k)}$, and return afterwards to the original ordering by multiplying with $P^{\ast}_{\sigma}$. In other words, the iteration step is now
\begin{eqnarray*}
y^{(k+1)} &=&P_{\sigma}^{\ast}[(I-\omega(I+\omega L_{\sigma})^{-1}B_{\sigma}) P_{\sigma} y^{(k)} +\omega (I+\omega L_{\sigma})^{-1}b_{\sigma}]\\
&=&
(\underbrace{I - \omega P_{\sigma}^{\ast}(I+\omega L_{\sigma})^{-1}P_{\sigma} B}_{=Q_{\sigma}})y^{(k)} +\omega P_{\sigma}^{\ast}(I+\omega L_{\sigma})^{-1}P_{\sigma} b.
\end{eqnarray*}
Thus, as before
$$
|e_\sigma^{(k+1)}|_B^2=|Q_{\sigma}(\bar{y}-y^{(k)})|_B^2 =|e^{(k)}|_B^2 - \omega(2-\omega)\|(I+\omega L_{\sigma})^{-1}P_{\sigma} B(e^{(k)})\|^2,
$$
where for short we have set  $e_\sigma^{(k+1)}:=\bar{y}-y^{(k+1)}$ and $e^{(k)}:=\bar{y}-y^{(k)}$ (indicating that $y^{(k+1)}$ depends on $\sigma$, while
$y^{(k)}$ is considered fixed at the moment).
The expected square semi-norm error after $k+1$ iterations (conditioned on the error $e^{(k)}$) is thus
\begin{equation}\label{Intermed}
\frac1{n!}\sum_{\sigma} |e_\sigma^{(k+1)}|_B^2 = |e^{(k)}|_B^2 - \frac{\omega(2-\omega)}{n!}\sum_{\sigma} \|(I+\omega L_{\sigma})^{-1}P_{\sigma} Be^{(k)}\|^2.
\end{equation}

We give a lower estimate for the last term in (\ref{Intermed}) with $Be^{(k)}$ temporarily replaced by any unit vector $z$. Since for positive
definite $S\in \mathscr H_n$ we have
$$
\langle Sy,y\rangle\langle S^{-1}y,y\rangle\ge 1,\qquad \|y\|=1,
$$  
(indeed, $1= \|y\|^4=\langle S^{1/2}y,S^{-1/2}y\rangle^2\le \|S^{1/2}y\|^2\|S^{-1/2}y\|^2$), applying this inequality with $S=(I+\omega L_{\sigma})(I+\omega L_{\sigma}^{\ast})$ and
$y=P_{\sigma} z$, we get 
$$
\|(I+\omega L_{\sigma})^{-1}P_{\sigma} z\|^{-2}\le \|(I+\omega L_{\sigma}^{\ast})P_{\sigma} z\|^{2} =\langle\|z\|^2 +\omega (Hz,z) + \omega^2 (\frac1{n!}\sum_{\sigma}P_{\sigma}^{\ast} L_{\sigma} L_{\sigma}^{\ast} P_{\sigma}) z, z\rangle 
$$
where as before $H=B-I=L+L^{\ast}$. Thus, by the arithmetic-harmonic-mean inequality,
\begin{eqnarray*}
\frac1{n!}\sum_{\sigma} \|(I+\omega L_{\sigma})^{-1}P_{\sigma} z\|^2
&\ge& n!\left(\sum_{\sigma} \|(I+\omega L_{\sigma})^{-1}P_{\sigma} z\|^{-2}\right)^{-1}\\
&\ge& n!\left(\sum_{\sigma} \|(I+\omega L_{\sigma}^{\ast})P_{\sigma} z\|^{2}\right)^{-1}\\
&=& \left( (\|z\|^2 +\omega (Hz,z) + \omega^2 (\frac1{n!}\sum_{\sigma}P_{\sigma}^{\ast} L_{\sigma} L_{\sigma}^{\ast} P_{\sigma} z, z) \right)^{-1}.
\end{eqnarray*}
By Theorem \ref{ThmSumM}, the sum in the last expression can be estimated by
$$
\|z\|^2 +\omega (Hz,z) + \omega^2 (\frac1{n!}\sum_{\sigma}P_{\sigma}^{\ast} L_{\sigma} L_{\sigma}^{\ast} P_{\sigma} z, z)\le (1+w\|H\|)^2 \le (1+\omega\lambda_1)^2.
$$
This gives the needed auxiliary result
$$
\frac1{n!}\sum_{\sigma} \|(I+\omega L_{\sigma})^{-1}P_{\sigma} z\|^2\ge (1+\omega\lambda_1)^{-2},\qquad \|z\|=1.
$$
Going back to the notation of (\ref{Intermed}), we therefore have
\begin{equation}\label{Intermed1}
\frac1{n!}\sum_{\sigma} |e^{(k+1)}_\sigma|_B^2 = |e^{(k)}|_B^2 - \frac{\omega(2-\omega)}{(1+\omega\lambda_1)^2}\|Be^{(k)}\|^2\le
\left(1-\frac{\omega(2-\omega)\lambda_r}{(1+\omega\lambda_1)^2}\right)|e^{(k)}|_B^2,
\end{equation}
which implies the desired estimate for the expected square energy semi-norm error after one iteration step, conditioned on the previous iterate. Since the random choice of
$\sigma$ is considered independent from iteration step to iteration step, we can take the expectation of $|e^{(k)}|_B^2$ in (\ref{Intermed1}) and arrive at the statement of Part a). Finally, we note that for singular $B$, the result of Part a) only implies that the unique solution component in $\mathrm{Ran}(B)$ is recovered at an exponential rate from the iterates (in expectation). 
\end{proof}

We conclude with a few further comments on the estimates for shuffled SOR iterations obtained in Theorem \ref{ThmSOR}. First of all, they are worst-case upper bounds for the class of all
consistent systems $By=b$ with Hermitian positive semi-definite matrix $B$ and normalization condition $D=I$. As such, they improve upon the worst-case
upper bounds for fixed cyclic ordering from Theorem \ref{theo1}, at least in the asymptotic regime $n\to \infty$. The current estimate $C_1\le 32.42$ entering
the bound for the preshuffled SOR iteration is certainly too pessimistic compared to our numerical experience reported in \cite{oswald2015convergence}, it is due to our reliance on Theorem \ref{LmPave} for which currently only suboptimal quantitative versions, i.e., crude estimates for $\gamma(\epsilon)$, are available. Finding better estimates for $\gamma(\epsilon)$ and the constant $C_1$ in (\ref{KSL}), or replacing the use
of simple norm estimates for $L$ by more subtle techniques, would be desirable. We leave this for future work.
\\[1ex]
Another issue is the formal superiority of the bound (\ref{SVest}) for the single-step randomized SOR iteration compared to our results which is not
reflected in the actual performance of the methods in many tests, where shuffled and preshuffled SOR iterations compete well.
The appearance of an additional factor $\lambda_1$ in the denominator of our convergence rate estimates in Theorems \ref{theo1} and \ref{ThmSOR} 
compared to (\ref{SVest}) is inherent to our approach of analyzing the error reduction per sweep rather than estimating the single-step error reduction. 
Due to the assumed normalization $D=I$,
we have $1\le \lambda_1\le n$, however in many practical cases (and for typical ensembles of random matrices) the actual value of $\lambda$ remains close to
$1$ which partly mitigates the issue. We conclude with an academic example showing that the extra $\lambda_1$ factor in the denominator of the bound in Theorem \ref{theo1} is necessary (whether this is also true for the bounds in Theorem \ref{ThmSOR} remains open).
\\[1ex]
For each $m\in \mathbb{N}$, consider the homogenous linear system $By=0$, where $B=AA^\ast$ is induced by the $2m\times 2$ matrix $A$
with unit norm row vectors $a_j$ given by
\begin{equation}
a_j=\begin{pmatrix}\cos((j-1)\theta_m), & \sin((j-1)\theta_m)\end{pmatrix},\quad j=1,\ldots, 2m,\qquad \theta_m:=\frac{\pi}{2m}.
\end{equation}
It is easy to check that $A^\ast A =m I$. Thus, $B$ has rank $r=2$, essential condition number $1$, and spectral norm $\|B\|=\lambda_1=m$. 
As mentioned in the introduction, applying the Gau{\ss}-Seidel method ($\omega=1$) to $By=0$ is the same as applying the Kaczmarz aka ADM method to $Ax=0$. From a geometric point of view (see the figure below), since the $2m$ hyperplanes the ADM method for $Ax=0$ projects on split the plane with equal angles $\theta_m$, the error reduction rate per single step of the ADM iteration with cyclic ordering is simply $\cos\theta_m$, and becomes increasingly slow as $m$ grows. The convergence rate of the squared error per sweep is thus 
$$
(\cos\theta_m)^{2m} \approx (1-\frac{\pi^2}{8m^2})^{2m} \approx 1-\frac{\pi^2}{4m},\qquad m\to \infty.
$$
This shows that without the $\lambda_1=m$ factor in the denominator of the bound (\ref{SORest}) from Theorem \ref{theo1} we would arrive at a contradiction.
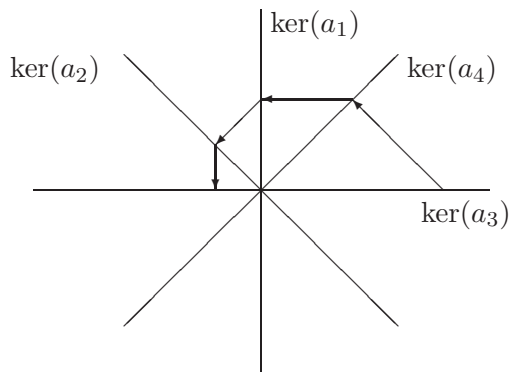
\begin{figure}[H]
\setlength{\unitlength}{0.6cm}
\begin{picture}(7,8.5)(-10,-4)
\put(3,0){\line(1,0){5}}
\put(3,0){\line(0,1){4}}
\put(3,0){\line(1,1){3}}
\put(3,0){\line(-1,1){3}}
\put(3,0){\line(-1,0){5}}
\put(6.5,-0.7){\text{$\ker(a_3)$}}
\put(6.2,2.5){\text{$\ker(a_4)$}}
\put(3.2,3.5){\text{$\ker(a_1)$}}
\put(-2.5,2.5){\text{$\ker(a_2)$}}
\put(7,0){\vector(-1,1){2}}
\put(5,2){\vector(-1,0){2}}
\put(3,2){\vector(-1,-1){1}}
\put(2,1){\vector(0,-1){1}}
\put(3,0){\line(0,-1){4}}
\put(3,0){\line(-1,-1){3}}
\put(3,0){\line(1,-1){3}}
\end{picture} 
\caption{Hyperplanes for ADM example with $m=4$} 
\label{FigurePlaneSplitting}
\end{figure}

\bibliographystyle{plain}
\bibliography{ShufflingExplanationPO1}

\begin{thebibliography}{10}

\bibitem{bourgaintzafriri1987}
J.~Bourgain and L.~Tzafriri.
\newblock Invertibility of ‘large’ submatrices with applications to the
  geometry of {B}anach spaces and harmonic analysis.
\newblock {\em Israel J. Math.}, 57(2):137--224, 1987.

\bibitem{casazzatremain2015}
P.~G. Casazza and J.~C. Tremain.
\newblock Consequences of the {M}arcus/{S}pielman/{S}rivastava solution of the
  {K}adison-{S}inger problem.
\newblock In {\em Trends in {A}pplied {H}armonic {A}nalysis}, pages 191--214.
  Springer, 2015.

\bibitem{davidson1988nest}
K.~R. Davidson.
\newblock {\em Nest Algebras}.
\newblock Longman Sci. Tech. Pub., 1988.

\bibitem{duchi2012}
J.~C. Duchi.
\newblock Commentary on "{T}owards a noncommutative arithmetic-geometric mean
  inequality" by {B}. {R}echt and {C}. {R}{\'e}.
\newblock {\em J. Mach. Learn. Res., W\&CP COLT 2012}, 23:11.25--11.27, 2012.

\bibitem{feichtinger1994theory}
H.~G. Feichtinger and K.~Gr{\"o}chenig.
\newblock Theory and practice of irregular sampling.
\newblock In {\em Wavelets: {M}athematics and {A}pplications}, Stud. Adv.
  Math., pages 305--363. CRC, Boca Raton, FL, 1994.

\bibitem{hackbusch1994iterative}
W.~Hackbusch.
\newblock {\em Iterative solution of large sparse systems of equations}.
\newblock Springer, 1994.

\bibitem{kashin1980}
B.~S. Kashin.
\newblock On some properties of matrices of bounded operators from the space
  $l_2^n$ into $l_2^m$ ({R}ussian).
\newblock {\em Izv. Akad. Nauk Arm. SSR, Mat.}, 15:379--394, 1980.
\newblock Engl. transl. in Soviet J. Contemporary Math. Anal. 15,5 (1980),
  44--57.

\bibitem{marcus2013interlacing}
A.~Marcus, D.~A. Spielman, and N.~Srivastava.
\newblock Interlacing families {II}: {M}ixed characteristic polynomials and the
  {K}adison-{S}inger problem.
\newblock {\em Ann. of Math. (2)}, 182(1):327--320, 2015.

\bibitem{mathias1993hadamard}
R.~Mathias.
\newblock The {H}adamard operator norm of a circulant and applications.
\newblock {\em SIAM J. Matrix Anal. Appl.}, 14(4):1152--1167, 1993.

\bibitem{oswald1994convergence}
P.~Oswald.
\newblock On the convergence rate of {SOR}: {A} worst case estimate.
\newblock {\em Computing}, 52(3):245--255, 1994.

\bibitem{oswald2015convergence}
P.~Oswald and W.~Zhou.
\newblock Convergence analysis for {K}aczmarz-type methods in a {H}ilbert space
  framework.
\newblock {\em Linear Algebra Appl.}, 478:131--161, 2015.

\bibitem{recht2012beneath}
B.~Recht and C.~R{\'e}.
\newblock Toward a noncommutative arithmetic-geometric mean inequality:
  conjectures, case-studies, and consequences.
\newblock {\em J. Mach. Learn. Res., W\&CP COLT 2012}, 23:11.1--11.24, 2012.

\bibitem{strohmer2009randomized}
T.~Strohmer and R.~Vershynin.
\newblock A randomized {K}aczmarz algorithm with exponential convergence.
\newblock {\em J. Fourier Anal. Appl.}, 15(2):262--278, 2009.

\bibitem{varga1959orderings}
R.~Varga.
\newblock Orderings of the successive overrelaxation scheme.
\newblock {\em Pacific J. Math.}, 9(3):925--939, 1959.

\bibitem{wright2015}
S.~J. Wright.
\newblock Coordinate descent algorithms.
\newblock {\em Math. Programming}, 151(1):3--34, 2015.

\bibitem{young1954iterative}
D.~Young.
\newblock Iterative methods for solving partial difference equations of
  elliptic type.
\newblock {\em Trans. Amer. Math. Soc.}, 76(1):92--111, 1954.

\end{thebibliography}

\end{document}